\documentclass[12pt,twoside,a4paper]{article}
\usepackage{amsmath,amssymb,amsthm}
\usepackage{hyperref}
\usepackage{graphicx,psfrag}

\newtheorem{thm}{Theorem}[section]
\newtheorem{lem}[thm]{Lemma}
\newtheorem{pro}[thm]{Proposition}
\newtheorem{cor}[thm]{Corollary}

\theoremstyle{definition}

\theoremstyle{remark}
\newtheorem{rem}[thm]{Remark}

\newcommand{\R}{\mathbb{R}}

\newcommand{\cB}{\mathcal{B}}

\newcommand{\cP}{\mathcal{P}}

\newcommand{\al}{\alpha}
\newcommand{\be}{\beta}
\newcommand{\ga}{\gamma}

\newcommand{\de}{\delta}

\newcommand{\om}{\omega}

\newcommand{\la}{\lambda}

\renewcommand{\phi}{\varphi}

\newcommand{\CAT}{\operatorname{CAT}}

\newcommand{\cd}{\operatorname{cd}}

\newcommand{\sing}{\operatorname{sing}}
\newcommand{\reg}{\operatorname{reg}}

\newcommand{\sign}{\operatorname{sign}}

\renewcommand{\d}{\partial}
\newcommand{\di}{\d_{\infty}}
\newcommand{\set}[2]{\{#1:\,\text{#2}\}}
\newcommand{\sm}{\setminus}
\newcommand{\sub}{\subset}

\newcommand{\ov}{\overline}
\newcommand{\wt}{\widetilde}
\newcommand{\wh}{\widehat}

\begin{document}

\title{M\"obius and sub-M\"obius structures}

\author{Sergei Buyalo\footnote{Supported by RFFI Grant
14-01-00062}}

\date{}
\maketitle

\begin{abstract} We introduce a notion of a sub-M\"obius
structure and find necessary and sufficient conditions under which 
a sub-M\"obius structure is a M\"obius structure. We show that 
on the boundary at infinity 
$\di Y$
of every Gromov hyperbolic space 
$Y$
there is a canonical sub-M\"obius structure which is invariant under
isometries of
$Y$
such that the sub-M\"obius topology on
$\di Y$
coincides with the standard one. 
\end{abstract}

\section{Introduction} A M\"obius structure on a set 
$X$
is a class of metrics having the same cross-ratio on every 4-tuple 
$P$
in
$X$
whenever it is well defined. There are six variants of cross-ratios, and
a permutation of entries of
$P$
induces a respective permutation of cross-ratios. We introduce a notion 
of a sub-M\"obius structure, a basic feature of which is the same
symmetry under permutations as in the case of M\"obius structures. 
The notion of a sub-M\"obius structure seems to be more flexible than 
that of a M\"obius structure. Like M\"obius structures any sub-M\"obius structure 
induces a uniquely determined topology on
$X$.
Our main result, Theorem~\ref{thm:submoeb_moeb}, establishes
necessary and sufficient conditions under which a sub-M\"obius structure is
a M\"obius structure. These conditions require to deal with admissible 5-tuples in
$X$,
and respective symmetries involve the symmetry group
$S_5$
(however, we do not touch the 
$S_5$-symmetry theme in this paper).

Furthermore, we show that on the boundary at infinity of
every Gromov hyperbolic space 
$Y$
there is a canonical sub-M\"obius structure which is invariant under
isometries of
$Y$
such that the sub-M\"obius topology on
$\di Y$
coincides with the standard one. A M\"obius structure with similar
properties which in addition is ptolemaic is known to exist if 
$Y$
is a hyperbolic complex, that is, a simplicial metric complex,
whose 1-skeleton is a hyperbolic graph with uniformly bounded
valence, see \cite{M}, or if
$Y$
is a
$\CAT(-1)$-space,
see \cite{B,FS}, or
$Y$
is asymptotically PT$_{-1}$, see \cite{MS}.

\section{Semi-metrics and the cross-ratio homomorphism}
\label{sect:semi-metrics}

Let
$X$
be a set. A function
$d:X^2\to\wh\R=\R\cup\{\infty\}$
is called a {\em semi-metric}, if it is symmetric,
$d(x,y)=d(y,x)$
for each
$x$, $y\in X$,
positive outside the diagonal, vanishes on the diagonal
and there is at most one infinitely remote point
$\om\in X$
for
$d$,
i.e. such that
$d(x,\om)=\infty$
for every
$x\in X\sm\{\om\}$.
A metric is a semi-metric that satisfies the triangle inequality.

A 4-tuple
$P=(x,y,z,u)\in X^4$
is said to be {\em admissible} if no entry occurs three or
four times in
$P$.
A 4-tuple
$P$
is {\em nondegenerate}, if all its entries are pairwise
distinct. In what follows, we always assume that
4-tuples we consider are admissible unless the opposite
is explicitly stated.

Let
$\cP_4=\cP_4(X)$
be the set of all ordered admissible 4-tuples of
$X$, $\reg\cP_4\sub\cP_4$
the set of nondegenerate 4-tuples.

We denote by
$L_4$
the subspace in
$\R^3$
given by
$a+b+c=0$
(subindex 
$4$
is related to 4-tuples rather than to the dimension of
$\R^3$).
We extend 
$L_4$
by adding to it thee points
$A=(0,\infty,-\infty)$, $B=(-\infty,0,\infty)$, $C=(\infty,-\infty,0)$,
$\ov L_4=L_4\cup\{A,B,C\}$.

With every semi-metric
$d$
we associate a map 
$M_d:\cP_4\to\ov L_4$
given by

\begin{equation}\label{eq:cross-ratio}
M_d(x_1,x_2,x_3,x_4)=(a,b,c), 
\end{equation}
where
\begin{itemize}
 \item[] $a=\cd(x_1,x_2,x_3,x_4)=(x_1|x_4)+(x_2|x_3)-(x_1|x_3)-(x_2|x_4)$
 \item[] $b=\cd(x_1,x_3,x_4,x_2)=(x_1|x_2)+(x_3|x_4)-(x_1|x_4)-(x_2|x_3)$
 \item[] $c=\cd(x_2,x_3,x_1,x_4)=(x_2|x_4)+(x_1|x_3)-(x_1|x_2)-(x_3|x_4)$,
\end{itemize}
$(x_i|x_j)=-\ln d(x_i,x_j)$
($\cd$
stays for {\em cross-difference}.) Note that
$a+b+c=0$,
and that if
$P=(x_1,x_2,x_3,x_4)$
is degenerate, then
$M_d(P)$
is well defined and takes values
$A$, $B$, $C$.
The map 
$M_d:\cP_4\to\ov L_4$
is called the {\em M\"obius structure} on
$X$
associated with
$d$.

\begin{rem}\label{rem:log_cross-ratio} We prefer to work with
logarithm of cross-ratios instead of cross-ratios themselves
by many reasons, most important of which is that the equations
(A), (B) of Theorem~\ref{thm:submoeb_moeb} are linear in the 
logarithmic setup. 
\end{rem}

Since
$(x|x)=\infty$,
we have, for example,
$M_d(x_1,x_1,x_3,x_4)=(0,\infty,-\infty)$.
Furthermore, since
$(x|\infty)=-\infty$,
we have e.g.
$M_d(x_1,x_2,x_3,\infty)=((x_2|x_3)-(x_1|x_3),(x_1|x_2)-(x_2|x_3),(x_1|x_3)-(x_1|x_2))$.
If in addition 
$x_3$
is also infinitely remote for 
$d$,
then
$M_d(x_1,x_2,\infty,\infty)=(0,\infty,-\infty)$.
Note that
$M_d(P)\in L_4$
if and only if
$P\in\reg\cP_4$
is nondegenerate.

A semi-metric
$d'$
on
$X$
is said to be {\em M\"obius equivalent} to
$d$,
if
$M_{d'}=M_d$.
For example, every {\em metric inversion} 
$d'$
of
$d$
is a semi-metric that is M\"obius equivalent to
$d$,
$$d'(x,y)=\frac{r^2d(x,y)}{d(x,\om)d(y,\om)},\ d'(\om,\om)=0,$$
for some
$r>0$
and all
$x$, $y\in X$
that are not equal to
$\om$
simultaneously, where
$\om\in X$
is not infinitely remote for
$d$.
Note that then
$\om$
is infinitely remote for
$d'$.
In particular, for every
$\om\in X$
there is a semi-metric
$d'$
equivalent
$d$
with infinitely remote
$\om$.
Actually, a M\"obius structure 
$M$
on
$X$
is a class M\"obius equivalent semi-metrics on
$X$,
and we write
$d\in M$
for a semi-metric
$d$
from this class.

We denote by
$S_k$
the symmetry group of
$k$
elements. We use notation 
$\pi(1\dots k)=\pi:=i_1\dots i_k$
for 
$\pi\in S_k$,
where
$i_j=\pi(j)$
for 
$j=1,\dots,k$.

The group
$S_4$
acts on
$\cP_4$
by permutations of the entries of every
$P\in\cP_4$.
The group
$S_3$
acts on
$\ov L_4$
by signed permutations of the coordinates. 
The {\em cross-ratio} homomorphism
$\phi:S_4\to S_3$
can be described as follows: a permutation of a tetrahedron
ordered vertices 
$(1,2,3,4)$
gives rise to a permutation of opposite pairs of edges
$((12)(34),(13)(24),(14)(23))$. 
Thus the kernel 
$K$
of
$\phi$
consists of four elements 1234, 2143, 4321, 3412,
and is isomorphic to the dihedral group
$D_4$
of a square automorphisms, which is also called the {\em Klein} group.
There are six left cosets of 
$K$, 
and the cosets have the following representatives:
\begin{align*}
 \phi(1234\cdot K)=\phi(K)&=123\\
 \phi(2134\cdot K)=\phi(1243\cdot K)&=132\\
 \phi(3214\cdot K)=\phi(1432\cdot K)&=321\\
 \phi(4231\cdot K)=\phi(1324\cdot K)&=213\\
 \phi(2431\cdot K)&=231\\ 
 \phi(3241\cdot K)&=312.
\end{align*}

One easily checks that the map
$M_d$
is equivariant with respect to the {\em signed} cross-ratio homomorphism
$\phi$,
i.e.
$$M_d(\pi P)=\sign(\pi)\phi(\pi)M_d(P)$$
for every
$P\in\cP_4$
and every
$\pi\in S_4$,
where
$\sign:S_4\to\{\pm 1\}$
the homomorphism that associates to every odd permutation the sign 
``$-1$''.

\section{Sub-M\"obius structures}

\subsection{Definition}
\label{subsect:definition}

A {\em sub-M\"obius structure} on a set
$X$
is given by a map
$M:\cP_4\to\ov L_4$,
which satisfies the following conditions
\begin{itemize}
 \item [(a)] $M$
is equivariant with respect to the signed cross-ratio homomorphism
$\phi$,
i.e.
$$M(\pi P)=\sign(\pi)\phi(\pi)M(P)$$
for every
$P\in\cP_4$
and every
$\pi\in S_4$.
 \item [(b)] $M(P)\in L_4$
if and only if
$P$
is nondegenerate, $P\in\reg\cP_4$;
 \item [(c)] $M(P)=(0,\infty,-\infty)$
for
$P=(x_1,x_1,x_3,x_4)\in\cP_4$.
\end{itemize}

It follows from (b) that
$M(P)\in\ov L_4\sm L_4$
if and only if
$P\in\cP_4$
is degenerate. Using (a) and (c) we conclude that

\begin{align*}
 M(x_1,x_1,x_3,x_4)&=M(x_1,x_2,x_3,x_3)=(0,\infty,-\infty)=A\\
 M(x_1,x_2,x_1,x_4)&=M(x_1,x_2,x_3,x_2)=(-\infty,0,\infty)=B\\
 M(x_1,x_2,x_3,x_1)&=M(x_1,x_2,x_2,x_4)=(\infty,-\infty,0)=C.
\end{align*}

\begin{lem}\label{lem:semi-metric_submoebius} For every semi-metric
$d$
on
$X$
the M\"obius structure
$M_d$
associated with 
$d$
is a sub-M\"obius structure.
\end{lem}

\begin{proof} It is straightforward to check properties
(a)--(c) for 
$M_d$.
\end{proof}

\subsection{Functions associated with a sub-M\"obius structure}

Let
$X$
be a set. Given an ordered tuple
$P=(x_1,\dots,x_k)\in X^k$
we use notation
$$P_i=(x_1,\dots,x_{i-1},x_{i+1},\dots,x_k)$$
for
$i=1,\dots,k$.

We define the set 
$\cP_n$
of the {\em admissible} 
$n$-tuples 
in
$X$, $n\ge 5$,
recurrently by asking for every
$P\in\cP_n$
the 
$(n-1)$-tuples
$P_i$, $i=1,\dots,n$,
to be admissible. The set 
$\sing\cP_n$
of 
$n$-tuples
$P\in\cP_n$
having equal entries is called the {\em singular}
subset of
$\cP_n$.
Its complement
$\reg\cP_n=\cP_n\sm\sing\cP_n$
is called the {\em regular} subset of
$\cP_n$. 
Note that 
$\reg\cP_n$
coincides with the set of nondegenerate 
$n$-tuples.

A triple
$A=(\al,\be,\om)\in X^3$
of pairwise distinct points is called a {\em scale triple}.
A 5-tuple
$P=(x,y,A)=(x,y,\al,\be,\om)$,
where
$A=(\al,\be,\om)$
is a scale triple, is admissible if and only if the cases
$x=y=\al$, $x=y=\be$, $x=y=\om$
are excluded. Then the 4-tuples
$$ P_1=(y,\al,\be,\om),\ P_2=(x,\al,\be,\om),\ P_3=(x,y,\be,\om),\ P_4=(x,y,\al,\om)$$
are admissible. For a sub-M\"obius structure 
$M$
on
$X$
we use notation
$M(P_i)=(a(P_i),b(P_i),c(P_i))$
for 
$i=1,\dots,4$.

\begin{lem}\label{lem:well_defined_A} Let
$M$
be a sub-M\"obius structure on
$X$, $A=(\al,\be,\om)\in X^3$
a scale triple,
$P=(x,y,A)\in\cP_5$.
Then
\begin{itemize}
 \item[] $b(P_1)+b(P_4)$
is well defined if and only if
$y\neq\al$;
 \item[] $-a(P_1)+b(P_3)$
is well defined if and only if
$y\neq\be$.
\end{itemize}
\end{lem}

\begin{proof} If
$y=\al$,
then
$b(P_1)=\infty$, $b(P_4)=-\infty$.
Hence
$b(P_1)+b(P_4)$
is not well defined. Similarly, if
$y=\be$,
then 
$a(P_1)=-\infty=b(P_3)$,
hence
$-a(P_1)+b(P_3)$
is not well defined.

In the opposite direction, assume
$y\neq\al$.
Then
$M(P_1)\in\ov L_4\sm L_4$
if and only if
$y=\be$
or
$y=\om$.
If
$y=\be$,
then 
$b(P_1)=0$,
and
$b(P_1)+b(P_4)$
is well defined. If
$y=\om$,
then 
$b(P_4)=0$,
and
$b(P_1)+b(P_4)$
is well defined.

Assume
$y\neq\be$.
Then 
$M(P_1)\in\ov L_4\sm L_4$
if and only if
$y=\al$
or
$y=\om$.
If
$y=\al$,
then 
$a(P_1)=0$,
and
$-a(P_1)+b(P_3)$
is well defined. If
$y=\om$,
then 
$b(P_3)=0$,
and
$-a(P_1)+b(P_3)$
is well defined.
\end{proof}

\begin{lem}\label{lem:well_defined_B} Let
$M$
be a sub-M\"obius structure on
$X$, $A=(\al,\be,\om)\in X^3$
a scale triple,
$P=(x,y,A)\in\cP_5$.
Then
$-a(P_4)+b(P_1)$
is well defined if and only if
$y\neq\al,\om$.
\end{lem}

\begin{proof} If
$y=\al$,
then
$b(P_1)=\infty$, $a(P_4)=\infty$.
Hence
$-a(P_4)+b(P_1)$
is not well defined. Similarly, if
$y=\om$,
then 
$b(P_1)=-\infty=a(P_4)$,
hence
$-a(P_4)+b(P_1)$
is not well defined.

In the opposite direction, assume
$y\neq\al,\om$.
Then
$M(P_1)\in\ov L_4\sm L_4$
if and only if
$y=\be$.
In this case
$b(P_1)=0$,
thus
$-a(P_4)+b(P_1)$
is well defined. 
\end{proof}

For a sub-M\"obius structure
$M$
on
$X$
we define functions which play an important role in a number
of questions, see sect.~\ref{subsect:submoebius_vs_moebius} and 
sect.~\ref{sect:submoebius_topology}. We let
$A=(\al,\be,\om)$
be a scale triple in
$X$.
In all cases when a 5-tuple
$P=(x,y,A)$
is not admissible, we put
$d_A^\al(x,y)=d_A^\be(x,y)=0$.
Assume that
$P=(x,y,A)$
is an admissible 5-tuple. We put
\begin{itemize}
 \item[] $d_A^\al(x,y)=e^{-b(P_1)-b(P_4)}$ 
if
$y\neq\al$,
 \item[] $d_A^\be(x,y)=e^{a(P_1)-b(P_3)}$ 
if
$y\neq\be$.
\end{itemize}
It follows from Lemma~\ref{lem:well_defined_A} that both
$d_A^\al(x,y)$, $d_A^\be(x,y)\in\wh\R$
are well defined. We use notation
$X_\om=X\sm\{\om\}$.

\begin{lem}\label{lem:def_semi-metric} The functions
$d_A^\al$, $d_A^\be$
defined above vanish on the diagonal,
$d_A^\al(x,x)=0=d_A^\be(x,x)$
for every
$x\in X$,
the point
$\om$
is infinitely remote for
$d_A^\al$, $d_A^\be$,
$d_A^\al(x,\om)=d_A^\be(x,\om)=\infty$
for every
$x\in X_\om$,
$d_A^\al(\om,x)=\infty$
for 
$x\neq\al,\om$, $d_A^\be(\om,x)=\infty$
for 
$x\neq\be,\om$,
and 
$d_A^\al(x,y),d_A^\be(x,y)\neq 0$
for each
$x\neq y$
whenever the functions are well defined.
\end{lem}

\begin{proof} Assume a 5-tuple
$P=(x,x,\al,\be,\om)$
is admissible. Then
$x\neq\al,\be,\om$
and thus
$M(P_1)\in L_4$, $M(P_3)=M(P_4)=(0,\infty,-\infty)$.
Hence
$d_A^\al(x,x)=0=d_A^\be(x,x)$
for all
$x\neq\al,\be,\om$. 
Using the definition
$d_A^\al(x,y)=0=d_A^\be(x,y)$
for the case when the 5-tuple
$P=(x,y,A)$
is not admissible, we obtain
$d_A^\al(x,x)=0=d_A^\be(x,x)$
for all
$x\in X$.

For the 5-tuple
$P=(x,\om,\al,\be,\om)$
we have
$P_1=(\om,\al,\be,\om)$, 
$P_4=(x,\om,\al,\om)$, 
$P_3=(x,\om,\be,\om)$. 
Thus
$M(P_1)=(\infty,-\infty,0)$,
$M(P_3)=M(P_4)=(-\infty,0,\infty)$.
If
$x\in X_\om$,
then
$P$
is admissible. We conclude that
$d_A^\al(x,\om)=\infty=d_A^\be(x,\om)$
for every
$x\in X_\om$.
Furthermore, the 5-tuple
$P=(\om,x,\al,\be,\om)$
with
$x\in X_\om$
is admissible and for
$P_4=(\om,x,\al,\om)$, $P_3=(\om,x,\be,\om)$
we have
$M(P_3)=M(P_4)=(\infty,-\infty,0)$.
Thus
$d_A^\al(\om,x)=\infty$
for 
$x\neq\al,\om$
and
$d_A^\be(\om,x)=\infty$
for 
$x\neq\be,\om$.

Assume
$d_A^\al(x,y)=0$
for an admissible 5-tuple
$P=(x,y,\al,\be,\om)$, $y\neq\al$.
Then
$M(P_1)\neq(0,\infty,-\infty)$,
hence
$b(P_4)=\infty$,
thus
$M(P_4)=(0,\infty,-\infty)$,
and we conclude that
$x=y$.

Assume
$d_A^\be(x,y)=0$
for an admissible 5-tuple
$P=(x,y,\al,\be,\om)$, $y\neq\be$.
Then
$M(P_1)\neq(-\infty,0,\infty)$,
hence
$b(P_3)=\infty$,
thus
$M(P_3)=(0,\infty,-\infty)$,
and we conclude that
$x=y$.
\end{proof}

\subsection{Sub-M\"obius structure vs M\"obius structure}
\label{subsect:submoebius_vs_moebius}

\begin{thm}\label{thm:submoeb_moeb} A sub-M\"obius structure
$M$
on
$X$
is a M\"obius structure if and only if for every scale triple
$A=(\al,\be,\om)\in X^3$
and every admissible 5-tuple
$P=(x,y,\al,\be,\om)$
the following conditions (A), (B) are satisfied
\begin{itemize}
 \item [(A)] $b(P_1)+b(P_4)=b(P_3)-a(P_1)$
for all
$y\in X$, $y\neq\al,\be$;
 \item[(B)] $b(P_2)=-a(P_4)+b(P_1)$,
where
$y\neq\al,\om$.
\end{itemize}
\end{thm}

It follows from Lemma~\ref{lem:well_defined_A} that 
the both sides of (A) are well defined,
and from Lemma~\ref{lem:well_defined_B} that 
the both  sides of the equality (B) are also well defined.

We prove  Theorem~\ref{thm:submoeb_moeb} in three steps. 
On the first step we define a function
$d=d_A:X\times X\to\wh\R$
for a fixed scale triple
$A=(\al,\be,\om)$.
Then Lemma~\ref{lem:def_semi-metric} implies that
$d$ 
vanishes on the diagonal, is positive outside of the diagonal, and
$\om$
is infinitely remote for
$d$.

On the second step, Lemma~\ref{lem:symmetric_dist}, we show that 
$d$
is symmetric,
$d(x,y)=d(y,x)$
for every
$x$, $y\in X$,
and therefore
$d$
is a semi-metric. On the final step, Lemmas~\ref{lem:moebius_restrict}, 
\ref{lem:inversion}, \ref{lem:extension}
and Corollary~\ref{cor:last_face}, we show that the M\"obius structure 
$M_d$
associated with
$d$
coincides with the initial sub-M\"obius structure,
$M_d=M$.

\begin{proof}[Proof of Theorem~\ref{thm:submoeb_moeb}]
It is straightforward to check that every M\"obius structure
$M$, $M=M_d$
for some semi-metric
$d$
on
$X$,
satisfies (A) and (B).

Assume that a sub-M\"obius structure
$M$
satisfies (A) and (B). We fix a scale triple
$A=(\al,\be,\om)\in X^3$.
One should think of
$\om$
as the infinitely remote point, and of the pair 
$\al$, $\be$
as a ``measuring rod'' which fixes the scale in
$X$.
We define a function
$d:X\times X\to\wh\R$, $d=d_A$,
as follows. We put
$d(\al,\al)=0$, $d(\be,\be)=0$, $d(\om,\om)=0$.
For
$x,y\in X$
such that the 5-tuple
$P=(x,y,\al,\be,\om)$
is admissible we put
\begin{equation}\label{eq:def_d_function}
 d(x,y)=\begin{cases}
d_A^\al(x,y)=e^{-b(P_1)-b(P_4)}&y\neq\al\\
d_A^\be(x,y)=e^{a(P_1)-b(P_3)}&y\neq\be
        \end{cases}.
\end{equation}
If
$y\neq\al,\be$,
then the right hand sides of (\ref{eq:def_d_function}) 
are well defined and they coincide by condition~(A).
By Lemma~\ref{lem:def_semi-metric} the function
$d$ 
vanishes on the diagonal, is positive outside of the diagonal, and
$\om$
is infinitely remote for
$d$.

\begin{lem}\label{lem:symmetric_dist} The function
$d=d_A$
is a semi-metric with 
$d(\al,\be)=1$
for every scale triple
$A=(\al,\be,\om)\in X^3$.
\end{lem}

\begin{proof} For an admissible 5-tuple
$P=(x,\al,\al,\be,\om)$
we have
$P_3=(x,\al,\be,\om)$,
$P_1=(\al,\al,\be,\om)$.
Thus
$a(P_1)=0$.
Hence using (\ref{eq:def_d_function}) we obtain
$$d(x,\al)=e^{-b(P_3)}=e^{-b(x,\al,\be,\om)}$$
for every
$x\neq\al$.
Furthermore, for the 5-tuple
$P=(\al,x,\al,\be,\om)$
we have
$P_4=(\al,x,\al,\om)$,
$P_1=(x,\al,\be,\om)$.
Thus
$b(P_4)=0$
and using (\ref{eq:def_d_function}) we obtain for
$x\neq\al$,
$$d(\al,x)=e^{-b(P_1)}=e^{-b(x,\al,\be,\om)}.$$
Thus
$d(x,\al)=d(\al,x)$
for all
$x\in X$.

Similarly, for an admissible 5-tuple
$P=(x,\be,\al,\be,\om)$
we have
$P_4=(x,\be,\al,\om)$,
$P_1=(\be,\al,\be,\om)$.
Thus
$b(P_1)=0$
and using (\ref{eq:def_d_function}) we obtain
$$d(x,\be)=e^{-b(P_4)}=e^{-b(x,\be,\al,\om)}$$
for every
$x\neq\be$.
Furthermore, for the 5-tuple
$P=(\be,x,\al,\be,\om)$
we have
$P_3=(\be,x,\be,\om)$,
$P_1=(x,\al,\be,\om)$.
Thus
$b(P_3)=0$
and using (\ref{eq:def_d_function}) we obtain for
$x\neq\be$,
$$d(\be,x)=e^{a(P_1)}=e^{a(x,\al,\be,\om)}.$$
Since
$b(x,\be,\al,\om)=-a(x,\al,\be,\om)$
by equivariance of
$M$
with respect to the signed cross-ratio homomorphism, we have
$d(x,\be)=d(\be,x)$
for all
$x\in X$.

Now, we show that
$d(x,y)=d(y,x)$.
By computation above, we can assume that the 5-tuples
$P=(x,y,\al,\be,\om)$, $P'=(y,x,\al,\be,\om)$
are nondegenerate and hence admissible.
By (\ref{eq:def_d_function}) we have
$d(y,x)=e^{-b(P_1')-b(P_4')}$.
By equivariance of
$M$
we have
$M(P_4')=(-a,-c,-b)$,
where
$M(P_4)=(a,b,c)$.
Thus
$b(P_4')=-c(P_4)$.
Using
$P_1'=P_2$
and condition~(B) we obtain
$-b(P_1')-b(P_4')=-b(P_2)+c(P_4)=a(P_4)-b(P_1)+c(P_4)=-b(P_1)-b(P_4)$
because
$a(P_4)+b(P_4)+c(P_4)=0$.
Hence
$d(y,x)=d(x,y)$.
Therefore, the function
$d:X\times X\to\wh\R$
is a semi-metric with infinite remote point
$\om$.

The 5-tuple 
$P=(\al,\be,\al,\be,\om)$
is admissible with
$P_1=(\be,\al,\be,\om)$, $P_4=(\al,\be,\al,\om)$.
Thus
$M(P_1)=(-\infty,0,\infty)=M(P_4)$
and
$b(P_1)=0=b(P_4)$.
By (\ref{eq:def_d_function}),
$d(\al,\be)=1$.
\end{proof}

It remains to check that
$M_d=M$.
The conditions~(a), (b), (c) of sect.~\ref{subsect:definition} are satisfied for
$M_d$.
Thus
$M_d(P)=M(P)$
for every degenerate
$P\in\cP_4$.

\begin{lem}\label{lem:moebius_restrict} Given a scale triple
$A=(\al,\be,\om)$
and an admissible 5-tuple
$P=(x,y,\al,\be,\om)$,
for the 4-tuples
$Q=P_1,P_2,P_3,P_4$
we have
$M_d(Q)=M(Q)$
for the semi-metric
$d=d_A$. 
\end{lem}

\begin{proof} We shall use that for
$Q=(x,y,z,\om)\in\cP_4$
$$M_d(Q)=((y|z)-(x|z),(x|y)-(y|z),(x|z)-(x|y))\in\ov L_4$$
because 
$\om$
is infinitely remote for 
$d$.
By remark before the lemma we can assume that 
the 4-tuples
$P_i$, $i=1,2,3,4$
are nondegenerate. 

For
$P_2=(x,\al,\be,\om)$
we put
$M(P_2)=(a,b,c)\in L_4$.
The 5-tuples
$P'=(x,\al,\al,\be,\om)$,
$P''=(x,\be,\al,\be,\om)$
are admissible and
$a(P_1')=0$, $b(P_3')=b(P_2)=b$, $b(P_1'')=0$, 
$b(P_4'')=b(x,\be,\al,\om)=-a(x,\al,\be,\om)=-a$.
By (\ref{eq:def_d_function}) we have
$$(x|\al)=-a(P_1')+b(P_3')=b,$$
and
$$(x|\be)=b(P_1'')+b(P_4'')=-a.$$
Thus
$M_d(P_2)=((\al|\be)-(x|\be),(x|\al)-(\al|\be),(x|\be)-(x|\al))
=(a,b,-a-b)=(a,b,c)=M(P_2)$.
Similarly, we obtain
$M_d(P_1)=M(P_1)$.
In particular,
$(x|\al)=b(P_2)$, $(x|\be)=-a(P_2)$, $(y|\al)=b(P_1)$, $(y|\be)=-a(P_1)$.

For
$P_4=(x,y,\al,\om)$
we have
\begin{align*}
M_d(P_4)&=((y|\al)-(x|\al),(x|y)-(y|\al),(x|\al)-(x|y))\\
        &=(b(P_1)-b(P_2),(x|y)-b(P_1),b(P_2)-(x|y)).
\end{align*}

Condition~(B) yields
$a(P_4)=b(P_1)-b(P_2)$.
By (2) we have
$(x|y)=b(P_1)+b(P_4)$,
thus
$b(P_4)=(x|y)-b(P_1)$.
Hence,
\begin{align*}
M(P_4)&=(a(P_4),b(P_4),c(P_4))\\
      &=(b(P_1)-b(P_2),(x|y)-b(P_1),c(P_4))=M_d(P_4).
\end{align*}
Similarly, we have
$M_d(P_3)=M(P_3)$
for 
$P_3=(x,y,\be,\om)$.
\end{proof}

\begin{lem}\label{lem:inversion} For scale triples
$A=(\al,\be,\om)$, $A'=(\be,\al,\om)$, $A''=(\al,\om,\be)$
we have
$d:=d_A=d_{A'}$
and
$d'=d_{A''}$
is the metric inversion of
$d$
with respect to 
$\be$
of radius 1,
$$d'(x,y)=\frac{d(x,y)}{d(x,\be)d(y,\be)}$$
for all
$x$, $y\in X$
that are not equal to
$\be$
simultaneously. In particular,
$d'$
is M\"obius equivalent to
$d$
and
$M_{d'}=M_d$.
\end{lem}

\begin{proof}
For the nondegenerate 5-tuples 
$P=(x,y,\al,\be,\om)$
and
$P'=(x,y,\al,\om,\be)$
we have
$P_1'=\pi\circ P_1$, $P_3'=\pi\circ P_3$,
where
$\pi$
is the permutation
$1243$.
Using that
$M(P_1)=M_d(P_1)$, $M(P_3)=M_d(P_3)$
by Lemma~\ref{lem:moebius_restrict}, we obtain
$$M(P_1')=M\circ\pi(P_1)=\sign(\pi)\phi(\pi)M(P_1)
  =\sign(\pi)\phi(\pi)M_d(P_1)=M_d(P_1')$$
and similarly
$M(P_3')=M_d(P_3')$.
Since
$\om$
is infinitely remote for 
$d$
and
$(\al|\be)=-\ln d(\al,\be)=0$,
we obtain
\begin{align*}
M(P_1')&=((y|\be),(y|\al)-(y|\be),-(y|\al))\\
M(P_3')&=((x|\be)-(y|\be),(x|y)-(x|\be),(y|\be)-(x|y)). 
\end{align*}
Then for the semi-metric
$d'=d_{A''}$
we have by (\ref{eq:def_d_function})
$$d'(x,y)=e^{a(P_1')-b(P_3')}=e^{(y|\be)-(x|y)+(x|\be)}
    =\frac{d(x,y)}{d(x,\be)d(y,\be)},$$
i.e.
$d'$
is the metric inversion of
$d$
with respect to
$\be$
of radius 1. The proof of the equality
$d_A=d_{A'}$
is similar.
\end{proof}

\begin{cor}\label{cor:last_face} Let
$d=d_A$
be the semi-metric associated with the scale triple
$A=(\al,\be,\om)\in X^3$,
and let
$P=(x,y,\al,\be,\om)$
be an admissible 5-tuple. Then
$M_d(P_5)=M(P_5)$.
\end{cor}

\begin{proof} We put
$P'=(x,y,\al,\om,\be)$
and note that
$M_{d'}(P_4')=M(P_4')$
for
$d'=d_{(\al,\om,\be)}$
by Lemma~\ref{lem:moebius_restrict}. Since
$M_d=M_{d'}$
by Lemma~\ref{lem:inversion}
and
$P_4'=(x,y,\al,\be)=P_5$,
we have
$M_d(P_5)=M_{d'}(P_4')=M(P_5)$.
\end{proof}

\begin{lem}\label{lem:extension} Given a nondegenerate 4-tuple
$(\al,\be,\om,o)\in\cP_4$,
the semi-metrics
$d_{(\al,\be,\om)}$
and
$d_{(\al,\be,o)}$
are M\"obius equivalent,
$d_{(\al,\be,\om)}\sim d_{(\al,\be,o)}$. 
\end{lem}

\begin{proof}
Using that
$d_{(\al,\be,\om)}\sim d_{(\al,\om,\be)}$
by Lemma~\ref{lem:inversion}, it suffices to show that the semi-metrics
$d=d_{(\al,\om,\be)}$
and
$d'=d_{(\al,o,\be)}$
are M\"obius equivalent.

We take 
$x$, $y\in X$
such that the 5-tuples
$P=(x,y,\al,\om,\be)$
and
$P'=(x,y,\al,o,\be)$
are nondegenerate. The 4-tuple
$P_1'=(y,\al,o,\be)$
is obtained from
$Q'=(y,o,\al,\be)=(y,o,\al,\be,\om)_5$
by the permutation
$\pi=1324$.
By Corollary~\ref{cor:last_face},
$M(Q')=M_{d_{(\al,\be,\om})}(Q')=M_d(Q')$.
Thus
$$M(P_1')=M_d(P_1')=((\al|o)-(y|o),(y|\al)-(\al|o),(y|o)-(y|\al))$$
because
$\be$
is infinitely remote for
$d$.
Note that
$P_4=(x,y,\al,\be)=P_4'$.
Hence
$$M(P_4')=M_d(P_4)=((y|\al)-(x|\al),(x|y)-(y|\al),(x|\al)-(x|y)).$$
Thus using (\ref{eq:def_d_function}) we obtain
$$d'(x,y)=e^{-b(P_1')-b(P_4'})=e^{(\al|o)-(x|y)}=\frac{d(x,y)}{d(\al,o)}.$$
Therefore
$d'=\la d$
for
$\la=1/d(\al,o)$,
and the semi-metrics
$d$
and
$d'$
are M\"obius equivalent.
\end{proof}

To complete the proof of Theorem~\ref{thm:submoeb_moeb}
it suffices to show that
$M_d(Q)=M(Q)$
for
$d=d_{(\al,\be,\om)}$
and any nondegenerate 4-tuple
$Q=(x,y,z,u)\in\reg\cP_4$
with all entries
$x$, $y$, $z$, $u$
different from
$\al$, $\be$, $\om$.
Applying Lemma~\ref{lem:inversion} and Lemma~\ref{lem:extension}
we have
$$d\sim d_{(\al,\om,\be)}\sim d_{(\al,\om,z)}\sim d_{(z,\al,\om)}.$$
Repeating we obtain
$d\sim d'$
for
$d'=d_{(z,u,\om)}$.
Hence
$M(Q)=M_{d'}(Q)=M_d(Q)$.
\end{proof}

\section{Sub-M\"obius structures and topology}
\label{sect:submoebius_topology}

Since a semi-metric 
$d$
on
$X$
not necessarily satisfies the triangle
inequality, the open balls in 
$X$
with respect to
$d$
not necessarily form a base of a topology on
$X$.
Nevertheless, we can still use the open balls to determine
a topology.

Recall that a collection 
$\cB$
of subsets of
$X$ 
is called a {\em subbase} of a topology
$\tau$
on
$X$
if finite intersections of elements of
$\cB$
form a base of
$\tau$.  
For any collection
$\cB$
of subsets there is a unique topology on
$X$
with subbase
$\cB$.

Let
$M$
be a M\"obius structure on a set
$X$.
For any scale triple
$A=(\al,\be,\om)\in X^3$
there is a uniquely determined semi-metric
$d_A$
with infinitely remote point
$\om$
such that 
$d_A(\al,\be)=1$.

The 
$M$-{\em topology}
on 
$X$
is determined by the subbase 
$\cB=\cB(M)$
consisting of all open balls,
$B_{A,t}(y)=\set{x\in X}{$d_A(x,y)<t$}$, 
where
$A=(\al,\be,\om)\in X^3$
is a scale triple,
$t>0$, $y\in X_\om=X\sm\{\om\}$.
Applying a metric inversion we see that complements 
of closed balls
$C_{A,t}(y)=\set{x\in X}{$d_A(x,y)>t$}$, $y\in X_\om$,
are also members of the subbase
$\cB$.
The $M$-topology associated with a M\"obius structure
$M$
is also called a {\em semi-metric topology}.

Every sub-M\"obius structure 
$M$
induces on
$X$
a canonical topology. Its definition is justified by 
Proposition~\ref{pro:semi-metric_topology} below and
Theorem~\ref{thm:submoeb_moeb}. 
We fix a scale triple
$A=(\al,\be,\om)$.
For every
$y\in X_\om$, $y\neq\al$
and
$t>0$
we set 
$$B_{A,t}^{\al}(y)=\set{x\in X}{$d_A^\al(x,y)<t$},$$
the open 
$\al$-ball
with respect to
$A$
of radius
$t$
centered at
$y$.
Note that
$\om\not\in B_{A,t}^\al(y)$
for any
$t>0$
because 
$d_A^\al(\om,y)=\infty$
by Lemma~\ref{lem:def_semi-metric}.

Similarly, for every
$y\in X_\om$, $y\neq\be$
and
$t>0$
we set
$$B_{A,t}^\be(y)=\set{x\in X}{$d_A^\be(x,y)<t$},$$
the open 
$\be$-ball
with respect to
$A$
of radius
$t$
centered at
$y$.
As above,
$\om\not\in B_{A,t}^\be(y)$.

The topology on
$X$
with the subbase 
$\cB=\cB(M)$
consisting of all open
$\al$-
and
$\be$-balls
with respect to all scale triples
$A=(\al,\be,\om)\in X^3$
is called the 
$M$-topology.
Given a sub-M\"obius structure
$M$
on
$X$, 
we always assume that 
$X$
is a topological space with the
$M$-topology.

\begin{lem}\label{lem:replace_scale} Given a scale triple
$A=(\al,\be,\om)\in X^3$,
for every
$y\neq\al,\om$
and every
$x\in X_\om$
such that the 5-tuple
$P=(x,y,A)$
is admissible, we have
$$d_A^\al(x,y)\cdot d_{A'}^\al(x,\om)=\la,$$
where
$A'=(\al,\om,y)$
and
$\la=\la(y,A)>0$
is independent of
$x$.

Furthermore, for every
$y\neq\be,\om$
and every
$x\in X_\om$
such that the 5-tuple
$P=(x,y,A)$
is admissible, we have
$$d_A^\be(x,y)\cdot d_{A''}^\be(x,\om)=\mu,$$
where
$A''=(\om,\be,y)$
and
$\mu=\mu(y,A)>0$
is independent of
$x$.
\end{lem}

\begin{proof} We assume that
$y\neq\al,\om$
and the 5-tuple
$P=(x,y,\al,\be,\om)$
is admissible with
$x\in X_\om$.
Then the 5-tuple
$P'=(x,\om,\al,\om,y)$
is admissible. We have
$M(P_1')=M(\om,\al,\om,y)=(-\infty,0,\infty)$,
thus
$b(P_1')=0$.
Furthermore
$M(P_4')=M(x,\om,\al,y)=M(\pi P_4)$
for 
$\pi=1432$.
Since
$\phi(\pi)=321$, $\sign(\pi)=-1$,
we obtain
$b(P_4)+b(P_4')=0$.
Therefore
$$d_A^\al(x,y)\cdot d_{A'}^\al(x,\om)
  =\exp\{-(b(P_1)+b(P_4)+b(P_1')+b(P_4'))\}=\la,$$
where
$\la=e^{-b(P_1)}>0$
depends only on
$y$, $A$.

The proof of the second equality is similar with
$P'=(x,\om,\om,\be,y)$
and
$\mu=e^{a(P_1)}$.
\end{proof}

For a scale triple
$A=(\al,\be,\om)\in X^3$
and every
$t>0$
we set 
$$C_{A,t}^\al(y)=\set{x\in X}{$d_A^\al(x,y)>t$},\ y\neq\al,\om$$
and
$$C_{A,t}^\be(y)=\set{x\in X}{$d_A^\be(x,y)>t$},\ y\neq\be,\om.$$

\begin{lem}\label{lem:complement_subbase} The sets
$C_{A,t}^\al(y)$
for 
$y\neq\al,\om$
and
$C_{A,t}^\be(y)$
for 
$y\neq\be,\om$
are members of the subbase 
$\cB=\cB(M)$.
\end{lem}

\begin{proof} First we note that for 
$y\neq\al,\om$
a 5-tuple
$P=(x,y,\al,\be,\om)$
is admissible for every
$x\in C_{A,t}^\al(y)$.
Indeed, if
$y\neq\be$,
then
$P$
is admissible for every
$x\in X$.
Assume that
$y=\be$.
Then
$P$
is not admissible only for 
$x=\be$.
In that case
$d_A^\al(x,y)=d_A^\al(\be,\be)=0$
by definition. Hence
$P$
is admissible for every
$x\in C_{A,t}^\al(y)$.

Assume
$y\neq\al,\om$.
Then for the scale triple
$A'=(\al,\om,y)$
and
$x\in X_\om$
we have by Lemma~\ref{lem:replace_scale}
$d_A^\al(x,y)>t$
if and only if
$d_{A'}^\al(x,\om)<\la/t$
for 
$\la=\la(y,A)>0$,
and
$d_A^\al(\om,y)=\infty$, $d_{A'}^\al(\om,\om)=0$.
Thus
$C_{A,t}^\al(y)=B_{A',\la/t}^\al(\om)$
is a member of
$\cB$.
 
Similarly, for 
$t>0$, $y\neq\be,\om$, 
and the scale triple
$A''=(\om,\be,y)$
the set 
$C_{A,t}^\be(y)=B_{A'',\mu/t}^\be(\om)$,
where
$\mu=\mu(y,A)>0$,
is a member of
$\cB$.
\end{proof}

\begin{cor}\label{cor:dist_continuity} For every scale triple
$A=(\al,\be,\om)\in X^3$
the functions
$x\mapsto d_A^\al(x,y), d_A^\be(x,y)$
are continuous on
$X$
whenever they are well defined.
\end{cor}

\begin{proof} The both functions take values in
$[0,\infty]$.
For any open interval
$I\sub[0,\infty]$
the inverse image is an intersection of two sets
$B_{A,s}^\al(y)\cap C_{A',t}^\al(y)$
for 
$A'=(\al,\om,y)$
or
$B_{A,s}^\be(y)\cap C_{A'',t}^\be(y)$
for 
$A''=(\om,\be,y)$
and appropriate 
$s$, $t\in[0,\infty]$
with 
$s>t$,
where the sets
$B_{A,s}^\al(y)$, $B_{A,s}^\be(y)$,
$C_{A',t}^\al(y)$, $C_{A'',t}^\be(y)$
coincide with 
$X$
for 
$s=\infty$
and
$t=0$.
By Lemma~\ref{lem:complement_subbase} these intersections are
open in the 
$M$-topology.
\end{proof}

\begin{pro}\label{pro:semi-metric_topology} Let
$M$
be a M\"obius structure on
$X$.
Then the semi-metric topology on
$X$
associated with 
$M$
coincides with the 
$M$-topology
associated with
$M$
regarded as a sub-M\"obius structure.
\end{pro}

\begin{proof} Let
$\cB_d$
be the subbase of the semi-metric topology,
$\cB$
the subbase of the 
$M$-topology
associated with
$M$
regarded as a sub-M\"obius structure.
We show that
$\cB_d=\cB$.

Recall, that for a scale triple
$A=(\al,\be,\om)\in X^3$
there is a uniquely determined semi-metric
$d$
with infinitely remote point
$\om$
such that 
$d(\al,\be)=1$.
We use notation
$(x|y)=-\ln d(x,y)$
for all
$x$, $y\in X$.

Consider an admissible 5-tuple
$P=(x,y,\al,\be,\om)$.
Then
\begin{align*}
M(P_1)&=(-(y|\be),(y|\al),(y|\be)-(y|\al)),\\ 
M(P_3)&=((y|\be)-(x|\be),(x|y)-(y|\be),(x|\be)-(x|y)),\\
M(P_4)&=((y|\al)-(x|\al),(x|y)-(y|\al),(x|\al)-(x|y)). 
\end{align*}

Thus for 
$y\neq\al,\om$
we have
$d_A^\al(x,y)=e^{-b(P_1)-b(P_4)}=d(x,y)$
and for 
$y\neq\be,\om$
we have
$d_A^\be(x,y)=e^{a(P_1)-b(P_3)}=d(x,y)$.
Hence open balls
$B_t(y)=B_{A,t}^\al(y)=B_{A,t}^\be(y)$.
Thus
$\cB_d=\cB$.
\end{proof}

\section{Hyperbolic spaces and sub-M\"obius structures}
\label{submoeb_hyperbolic}

\subsection{Canonical sub-M\"obius structure on $\di Y$} 
\label{subsect:canonical_submoeb}
Let
$Y$ 
be a Gromov hyperbolic space. We describe a canonical
sub-M\"obius structure on its boundary at infinity
$X=\di Y$.
Recall that the Gromov product of points
$y$, $y'\in Y$
with respect to a base point
$o\in Y$
is defined by
$$(y|y')_o=\frac{1}{2}(|yo|+|y'o|-|yy'|),$$
where
$|yy'|$
is the distance between
$y$, $y'$.

Given an admissible 4-tuple
$P=(\al,\be,\ga,\de)\in X^4$,
we take sequences
$\{a_i\}\in\al$, $\{b_i\}\in\be$, $\{c_i\}\in\ga$, $\{d_i\}\in\de$,
fix 
$o\in Y$
and consider the cross-difference
$$\cd(a_i,b_i,c_i,d_i)=(a_i|d_i)_o+(b_i|c_i)_o-(a_i|c_i)_o-(b_i|d_i)_o$$
(if some entry occurs in
$(\al,\be,\ga,\de)$
twice, then in both cases we take one and the same sequence representing
the entry). This expressions is independent of 
$o$,
and we put
$$\cd(\al,\be,\ga,\de)=\sup\liminf_i\cd(a_i,b_i,c_i,d_i),$$
where the supremum is taken over all sequences
$\{a_i\}\in\al$, $\{b_i\}\in\be$, $\{c_i\}\in\ga$, $\{d_i\}\in\de$.
Note that if
$(\al,\be,\ga,\de)$
is degenerate, then 
$\lim_i\cd((a_i,b_i,c_i,d_i)$
takes values
$0,\pm\infty$
and is well defined. Thus
$\cd(\al,\be,\ga,\de)$
is well defined, and it takes infinite values
$\pm\infty$
only if
$(\al,\be,\ga,\de)$
degenerates. For example,
$\cd(\al,\al,\ga,\de)=0$,
$\cd(\al,\be,\al,\de)=-\infty$
and
$\cd(\al,\be,\ga,\al)=\infty$.

Now, we put
$$\wt M(P)=(\wt u,\wt v,\wt w)\in\ov L_4,$$
where
$\wt u=\cd(\al,\be,\ga,\de)$, $\wt v=\cd(\al,\ga,\de,\be)$, 
$\wt w=-\wt u-\wt v$.
By properties of the Gromov product in hyperbolic spaces we have
$\wt M(P)\in L_4$
for every nondegenerate
$P$, $P\in\reg\cP_4$,
where
$\cP_4=\cP_4(X)$.
In the case
$P=(\al,\be,\ga,\de)$
is degenerate, we have by discussion above
\begin{align*}
 \wt M(\al,\al,\ga,\de)&=\wt M(\al,\be,\ga,\ga)=(0,\infty,-\infty)\\
 \wt M(\al,\be,\al,\de)&=\wt M(\al,\be,\ga,\be)=(-\infty,0,\infty)\\
 \wt M(\al,\be,\ga,\al)&=\wt M(\al,\be,\be,\de)=(\infty,-\infty,0),
\end{align*}
in particular,
$$\wt M(\pi(P))=\sign(\pi)\phi(\pi)\wt M(P)$$
for every degenerate
$P\in\cP_4$
and every
$\pi\in S_4$.

Finally, we define
$$M(P):=\frac{1}{|S_4|}\sum_{\rho\in S_4}\sign(\rho)\phi(\rho^{-1})\wt M(\rho(P)).$$
Then
$M:\cP_4\to\ov L_4$
is equivariant with respect to the signed cross-ratio homomorphism,
$$M(\pi(P))=\sign(\pi)\phi(\pi)M(P)$$
for every
$P\in\cP_4$, $\pi\in S_4$,
and
$M(P)=\wt M(P)$
for every degenerate
$P\in\cP_4$.
Therefore 
$M$
is a sub-M\"obius structure on
$X=\di Y$,
which is called the {\em canonical} sub-M\"obius structure
associated with
$Y$.
Though 
$M$
is defined by fixing
$o\in Y$,
it is independent of the choice of
$o$
and hence
$M$
is invariant under the action of any isometry group of
$Y$, $M(\ga(P))=M(P)$
for every isometry
$\ga:Y\to Y$
extended to
$\di Y$
and every
$P\in\cP_4$.

\subsection{Subbases of the standard topology on $\di Y$}

The Gromov product of points
$\al$, $\be\in X$
with respect to
$o\in Y$
is defined as
$$(\al|\be)_o=\inf\liminf_i(a_i|b_i)_o,$$
where the infimum is taken over all sequences
$\{a_i\}\in\al$, $\{b_i\}\in\be$.
For 
$x\in X$, $t>0$
the sets
$$U_{t,o}(x)=\set{y\in X}{$(x|y)_o>t$}$$
form for every
$o\in Y$
the base of neighborhoods of
$x$
in the {\em standard} topology on
$X=\di Y$
with subbase
$\cB_o=\set{U_{t,o}(x)}{$x\in X,t>0$}$.
This topology does not depend on the choice of
$o$.

Instead of the Gromov product on
$X$
with respect to
$o\in Y$,
it is convenient to use the Gromov product with respect to a scale triple
$A=(\al,\be,\om)\in X^3$,
which is defined as follows
$$(x|y)_{A,o}:=(x|y)_\om-(\al|\be)_\om,$$
where
$(x|y)_\om:=(x|y)_o-(x|\om)_o-(y|\om)_o$.
This corresponds to the metric inversion
$$d_{A,o}(x,y)=r^2d_\om(x,y)=\frac{r^2d_o(x,y)}{d_o(x,\om)d_o(y,\om)}$$
with
$r^2=1/d_\om(\al,\be)$,
where
$d_o(x,y)=e^{-(x|y)_o}$, $d_{A,o}(x,y)=e^{-(x|y)_{A,o}}$.

\begin{lem}\label{lem:metric_inversion_standard_topology} For every
$o\in Y$
the sets 
$U_{A,t,o}(y)=\set{x\in X}{$(x|y)_{A,o}>t$}$, $A=(\al,\be,\om)\in X^3$
is a scale triple,
$y\in X_\om$, $t>0$,
form a subbase of the standard topology on
$X=\di Y$.
\end{lem}

\begin{proof} The Gromov product
$(x|x')_o$
on
$X$
satisfies
$h$-inequality,
where
$h\ge 0$
is the hyperbolicity constant of
$Y$,
that is, for arbitrary
$x$, $y$, $z\in X$
two lowest members of the triple
$(x|y)_o$, $(x|z)_o$, $(y|z)_o$
differ from each other at most by
$h$,
see \cite[Lemma~2.2.2]{BS}. Thus for a fixed scale triple
$A=(\al,\be,\om)\in X^3$, $y\in X_\om$
and 
$t>(y|\om)_o+h$
we have
$|(x|\om)_o-(y|\om)_o|\le h$
for every
$x\in X$
with
$(x|y)_o>t$.
It follows that 
$$(x|y)_{A,o}=(x|y)_o-(x|\om)_o-(y|\om)_o-(\al|\be)_\om\ge (x|y)_o-c$$
for 
$c=2(y|\om)_o+|(\al|\be)_\om|+h$.
Since
$(x|y)_{A,o}\le(x|y)_o-(\al|\be)_\om\le(x|y)_o+c$,
we obtain
\begin{equation}\label{eq:subbase_compare}
U_{A,t+c,o}(y)\sub U_{t,o}(y)\sub U_{A,t-c,o}(y)
\end{equation}
for every
$t>(y|\om)_o+h$.

Let
$\tau$
be the standard topology on
$X$, $\tau'$
the topology with the subbase formed by the sets
$U_{A,t,o}$.
By Corollary~\ref{cor:dist_continuity}, the function
$x\mapsto d_{A,o}(x,y)$
is continuous on
$X$
in
$\tau'$
for every
$y\in X_\om=X\sm\{\om\}$.
In particular,
$(z|y)_{A,o}\to (x|y)_{A,o}$
as
$z\to x$
in
$\tau'$.
Thus for every
$x\in U_{A,t,o}(y)$
there is
$s>t$
such that 
$U_{A,s,o}(x)\sub U_{A,t,o}(y)$.
Using (\ref{eq:subbase_compare}), we obtain
$U_{s+c,o}(x)\sub U_{A,t,o}(y)$.
That is, every set in
$X$
open in
$\tau'$
is open in the standard topology. 

In the opposite direction, let
$t_0=\inf(x|y)_o$,
where the infimum is taken over all
$x$, $y\in X$.
Given
$y\in X$, $t>t_0+h$,
we take a scale triple
$A=(\al,\be,\om)\in X^3$
with 
$t>(y|\om)_o+h$.
Then 
$\om\not\in U_{t,o}(y)$.
The function
$x\mapsto(x|y)_o$
on
$X$
is continuous in
$\tau$,
thus for every
$x\in U_{t,o}(y)$
there is
$s>t$
such that 
$U_{s,o}(x)\sub U_{t,o}(y)$.
Using (\ref{eq:subbase_compare}), we obtain
$U_{A,s+c,o}(x)\sub U_{t,o}(y)$,
where
$c=|(\al|\be)_\om|$.
This shows that
$\tau'=\tau$.
\end{proof}

\subsection{Sub-M\"obius structure $M$ and the standard topology}

\begin{pro}\label{pro:canonical_submoeb_hyperbolic} Let
$Y$
be a Gromov hyperbolic space. For the canonical
sub-M\"obius structure
$M$
on
$X=\di Y$, 
the 
$M$-topology
coincides with the standard topology of
$X$.
\end{pro}

\begin{rem}\label{rem:boundary_continuous} Assume
$Y$
is {\em boundary continuous}, that is for every
$\al$, $\be\in\di X$, $o\in X$
and arbitrary sequences
$\{a_i\}\in\al$, $\{b_i\}\in\be$
there exists the limit
$\lim_i(a_i|b_i)_o=(\al|\be)_o$.
For example, every
$\CAT(-1)$
is boundary continuous, see \cite[sect.~3.4.2]{BS}. It follows that 
$$\cd(\al,\be,\ga,\de)=(\al|\de)_o+(\be|\ga)_o-(\al|\ga)_o-(\be|\de)_o$$
for every admissible 4-tuple
$P=(\al,\be,\ga,\de)\sub X$.
Then
$\wt M:\cP_4\to\ov L_4$,
where
$\wt M$
is defined in sect.~\ref{subsect:canonical_submoeb},
is equivariant with respect to the signed cross-ratio homomorphism
and hence the canonical sub-M\"obius structure
$M=\wt M$
is actually M\"obius structure associated with the semi-metric
$d_o(\al,\be)=e^{-(\al|\be)_o}$.
In this case the semi-metric topology on
$X$
coincides with the standard topology, and 
Proposition~\ref{pro:canonical_submoeb_hyperbolic} follows from
Proposition~\ref{pro:semi-metric_topology}. In a general case, there is
no reason to assume that
$Y$
is boundary continuous, and respective complications are discussed
in the proof below.
\end{rem}

\begin{rem}\label{rem:o_moeb_structure} In general case, we define for 
a fixed
$o\in Y$
the cross-difference
$$\cd_o(\al,\be,\ga,\de):=(\al|\de)_o+(\be|\ga)_o-(\al|\ga)_o-(\be|\de)_o$$
for any admissible 4-tuple
$P=(\al,\be,\ga,\de)$
in
$X$.
Setting
$M_o(P)=(u,v,w)$,
where
\begin{align*}
u&=\cd_o(\al,\be,\ga,\de)\\ 
v&=\cd_o(\al,\ga,\de,\be)\\
w&=\cd_o(\be,\ga,\al,\de), 
\end{align*}
we obtain a M\"obius structure 
$M_o$
on
$X$
associated with the semi-metric
$d_o(x,y)=e^{-(x|y)_o}$.
The respective semi-metric topology on
$X$
is the standard topology. However, disadvantage is that 
$M_o$
depends on
$o$,
and there is no reason for isometries of
$Y$
to preserve
$M_o$,
which leads to a quasification.
\end{rem}

\begin{lem}\label{lem:submoeb_omoeb} Let
$Y$
be 
$h$-hyperbolic
with 
$h\ge 0$.
Then for every admissible 4-tuple
$P\in\cP_4$
we have
$|M(P)-M_o(P)|\le 10h$.
\end{lem}

\begin{proof} Assume
$P=(\al,\be,\ga,\de)$.
Using that 
$$(\al|\be)_o\le\liminf_i(a_i|b_i)_o\le\limsup_i(a_i|b_i)_o\le(\al|\be)_o+2h$$
for arbitrary sequences
$\{a_i\}\in\al$, $\{b_i\}\in\be$,
see \cite[Lemma~2.2.2]{BS},
we obtain
\begin{align*}
\liminf_i\cd(a_i,b_i,c_i,d_i)&\le\limsup_i[(a_i|d_i)_o+(b_i|c_i)_o]
 -\liminf_i[(a_i|c_i)_o+(b_i|d_i)_o]\\
 &\le(\al|\de)_o+(\be|\ga)_o+4h-(\al|\ga)_o-(\be|\de)_o\\
 &\le\cd_o(\al,\be,\ga,\de)+4h 
\end{align*}
for arbitrary sequences
$\{a_i\}\in\al$, $\{b_i\}\in\be$, $\{c_i\}\in\ga$, $\{d_i\}\in\de$,
and hence
$\cd(\al,\be,\ga,\de)\le\cd_o(\al,\be,\ga,\de)+4h$. 

In the opposite direction, we have
$$(\al|\de)_o+(\be|\ga)_o\le\liminf_i[(a_i|d_i)_o+(b_i|c_i)_o]$$
and
$$(\al|\ga)_o+(\be|\de)_o\ge\limsup_i[(a_i|c_i)_o+(b_i|d_i)_o]-4h.$$
Hence
\begin{align*}
\cd_o(\al,\be,\ga,\de)&\le\liminf_i[(a_i|d_i)_o+(b_i|c_i)_o]
                         -\limsup_i[(a_i|c_i)_o+(b_i|d_i)_o]+4h\\
                      &\le\liminf_i\cd(a_i,b_i,c_i,d_i)+4h 
\end{align*}
for arbitrary sequences
$\{a_i\}\in\al$, $\{b_i\}\in\be$, $\{c_i\}\in\ga$, $\{d_i\}\in\de$.
Thus
$\cd_o(\al,\be,\ga,\de)\le\cd(\al,\be,\ga,\de)+4h$.
It follows that
$$|M_o(P)-\wt M(P)|^2\le|u-\wt u|^2+|v-\wt v|^2+|\wt u+\wt v-(u+v)|^2\le 96h^2,$$
where
$M_o(P)=(u,v,w)$, $\wt M(P)=(\wt u,\wt v,-\wt u-\wt v)$.
Since
$M_o$
is equivariant with respect to the signed cross-ratio homomorphism, we have
$$M(P)-M_o(P)=\frac{1}{|S_4|}\sum_{\rho\in S_4}
              \sign(\rho)\phi(\rho^{-1})(\wt M(\rho(P))-M_o(\rho(P))).$$
Hence
$|M(P)-M_o(P)|\le h\sqrt{96}<10h$
for every
$P\in\cP_4$.
\end{proof}

\begin{proof}[Proof of Proposition~\ref{pro:canonical_submoeb_hyperbolic}] For 
$o\in Y$, $t>0$,
a scale triple
$A=(\al,\be,\om)\in X^3$
and
$y\in X_\om$
we denote
$B_{A,t,o}(y)=\set{x\in X}{$d_{A,o}(x,y)<t$}$.

Then
$B_{A,t,o}(y)=U_{A,\ln\frac{1}{t},o}(y)$
because
$d_{A,o}(x,y)=e^{-(x|y)_{A,o}}$.
Thus by Lemma~\ref{lem:metric_inversion_standard_topology}
the sets
$B_{A,t,o}(y)$, $A=(\al,\be,\om)\in X^3$
is a scale triple,
$y\in X_\om$, $t>0$,
form a subbase of the standard topology on
$X=\di Y$.

Given an admissible 5-tuple
$P=(x,y,\al,\be,\om)\in X^4$
with a scale triple
$A=(\al,\be,\om)$,
we have 
\begin{align*}
a_o(P_1)&=M_o(y,\al,\be,\om)=(y|\om)_o+(\al|\be)_o-(y|\be)_o-(\al|\om)_o\\ 
b_o(P_1)&=M_o(y,\be,\om,\al)=(y|\al)_o+(\be|\om)_o-(y|\om)_o-(\al|\be)_o\\
b_o(P_3)&=M_o(x,\be,\om,y)=(x|y)_o+(\be|\om)_o-(x|\om)_o-(\be|y)_o\\
b_o(P_4)&=M_o(x,\al,\om,y)=(x|y)_o+(\al|\om)_o-(x|\om)_o-(\al|y)_o.
\end{align*}
Hence
$b_o(P_1)+b_o(P_4)=b_o(P_3)-a_o(P_1)=(x|y)_{A,o}$.
Thus using Lemma~\ref{lem:submoeb_omoeb}, we obtain
$$|b(P_1)+b(P_4)-(x|y)_{A,o}|\le 20h\quad\text{for}\quad y\neq\al$$
and  
$$|b(P_3)-a(P_1)-(x|y)_{A,o}|\le 20h\quad\text{for}\quad y\neq\be,$$
where
$M(P_i)=(a(P_i),b(P_i),c(P_i))$.
It follows that for 
$y\in X_\om$
we have
$$B_{A,t/20h,o}(y)\sub B_{A,t,o}^\al(y)\sub B_{A,20ht,o}(y)\quad\text{for}\quad y\neq\al$$
and
$$B_{A,t/20h,o}(y)\sub B_{A,t,o}^\be(y)\sub B_{A,20ht,o}(y)\quad\text{for}\quad y\neq\be.$$

Let
$B_{A,t,o}(y)$
be a subbase set of the standard topology
with a scale triple
$A=(\al,\be,\om)$, $y\in X_\om$, $t>0$.
By Corollary~\ref{cor:dist_continuity}, the function
$x\mapsto d_{A,o}(x,y)$
is continuous, thus for every
$x\in B_{A,t,o}(y)$
there is
$0<s<t$
such that 
$B_{A,s,o}(x)\sub B_{A,t,o}(y)$.
Since
$B_{A,s/20h,o}^\al(x)\sub B_{A,s,o}(x)$,
we see that every set in
$X$,
which is open in the standard topology, is open in the 
$M$-topology.
By a similar argument the converse is also true and hence the 
$M$-topology
coincides with the standard one.
\end{proof}


\bigskip
\begin{tabbing}

Sergei Buyalo\\

St. Petersburg Dept. of Steklov Math. Institute RAS,\\ 
Fontanka 27, 191023 St. Petersburg, Russia\\
{\tt sbuyalo@pdmi.ras.ru}\\

\end{tabbing}

\end{document}